\documentclass[11pt]{amsart}
\usepackage{amsmath,amssymb,amsthm,mathtools}
\usepackage[T1]{fontenc}
\usepackage{microtype}
\usepackage{hyperref}

\hypersetup{
  colorlinks=true,
  linkcolor=blue,
  citecolor=blue,
  urlcolor=blue
}

\newtheorem{theorem}{Theorem}[section]
\newtheorem{lemma}[theorem]{Lemma}
\newtheorem{proposition}[theorem]{Proposition}
\newtheorem{corollary}[theorem]{Corollary}
\newtheorem{question}[theorem]{Question}
\theoremstyle{remark}
\newtheorem{remark}[theorem]{Remark}
\theoremstyle{definition}

\newcommand{\Tr}{\operatorname{Tr}}
\newcommand{\rank}{\operatorname{rank}}

\newcommand{\id}{\mathrm I}

\title[Sharp Quasi-Reverse Minkowski Inequality]{Sharp Quasi-Reverse Minkowski Inequality for Schatten Norms}

\author{Hongsen Qiu}
\address{Institute for Advanced Study in Mathematics, Harbin Institute of Technology, Harbin, 150001, CHINA}
\email{chieughongsen@gmail.com}


\begin{document}

\begin{abstract}
Let $\|\cdot\|_p$ denote the Schatten $p$-norm and let
$|A|=(A^*A)^{1/2}$.  For $2\leq p<\infty$, let $x_{p,m}>1$ be the
unique solution of $x_{p,m}^p=2x_{p,m}+m-1$, and set
\[
 C_{p,m}=\frac{\sqrt{x_{p,m}(x_{p,m}+m-1)}}{(x_{p,m}^p+m-1)^{1/p}}.
\]
We prove the sharp inequality
\[
 \|A_1+\cdots+A_m\|_p\leq C_{p,m}\bigl\||A_1|+\cdots+|A_m|\bigr\|_p
\]
for arbitrary complex matrices of arbitrary size.  Equivalently, if
$q=p/(p-1)$ and $R,X_1,\cdots,+X_m$ are positive semidefinite, then
\[
 \|RX_1\|_1+\cdots\|RX_m\|_1
 \leq C_{p,m}\|R\|_q\|X_1+\cdots X_m\|_p.
\]
For $1<p<2$, we also show that the formula proposed for the
optimal constant fails. We give both a numerical counterexample and a systematic analytic construction.
\end{abstract}

\maketitle

\section{Introduction}
For a complex matrix $A$, write
\[
 |A|=(A^*A)^{1/2}.
\]
Tang and Zhang asked for the optimal constant in
\[
 \|A_1+A_2+\cdots+A_m\|_p\leq C_{p,m}\bigl\||A_1|+|A_2|+\cdots+|A_m|\bigr\|_p,\quad p\geq 1.
\]
The $m=2,\,p=2$ case was conjectured in \cite{Lee2010} and proved in
\cite{LinZhang2022}; further proofs were given in
\cite{Zamani2023,Zhang2025}.  Tang and Zhang
\cite[Conjecture~3.1]{TangZhang2025} subsequently formulated the
optimal-constant problem for $m$ summands and proposed a closed formula
for $C_{p,m}$, which is
\begin{equation}\label{eq:Cp}
 C_{p,m}=\frac{\sqrt{x_{p,m}(x_{p,m}+m-1)}}{(x_{p,m}^p+m-1)^{1/p}},
 \quad x_{p,m}>1,\quad x_{p,m}^p=2x_{p,m}+m-1,
\end{equation}
where $C_{\infty,m}=\sqrt{m}$ by continuity. They verified the conjecture at $p=1,2,\infty$ and obtained a
general upper bound, but left the intermediate exponents open. One can also see the recent work \cite{Zhang2026}. 

In this paper, our
 main result confirms their formula for the
 range $p\geq2$.  However, in the range $1<p<2$,
 Section~\ref{sec:below-two} shows that the proposed formula fails for
 every exponent, explains a systematic counterexample
 construction, and records our observations. A concurrent work~\cite{ZengLiu2026} independently obtained a similar counterexample and proved the result for $m=2,\,p=4$.

\begin{theorem}\label{thm:main}
Let $2\leq p\leq \infty$ and $\mathcal{H}$ be a finite-dimensional Hilbert space of arbitrary dimensions. Let $A_1,\cdots,A_m\in \mathcal{B}(\mathcal{H})$, then
\begin{equation}\label{eq:main}
 \|A_1+\cdots+A_m\|_p\leq C_{p,m}\bigl\||A_1|+\cdots+|A_m|\bigr\|_p.
\end{equation}
The constant $C_{p,m}$ is given by
\eqref{eq:Cp} and is optimal over $\dim(\mathcal{H})<\infty$.
\end{theorem}
\begin{remark}
For the case $m=1$ or ${\rm dim}(\mathcal{H})=1$, the result trivially holds. Hence, throughout the paper we only discuss the case where $m\geq 2$ and ${\rm dim}(\mathcal{H})\geq 2$.
\end{remark}
\section{Preliminaries}
In this paper, we consider a finite-dimensional Hilbert space $\mathcal{H}$ and all the bounded linear operators $\mathcal{B}(\mathcal{H})$. We denote the set of all the positive semi-definite operators as $\mathcal{B}(\mathcal{H})_+$.

For $\mathcal{H}$, using the bra-ket notation, one can fix a complete orthonormal basis $\{|e_1\rangle,|e_2\rangle,\cdots,|e_d\rangle\}$. Moreover, for each $A\in \mathcal{B}(\mathcal{H})$, it has a faithful matrix representation
$$ A=\sum_{i,j=1}^{d} A_{i,j}|e_i\rangle \langle e_j|,\quad A_{i,j}\in \mathbb{C}.$$
Hence we get $\mathcal{B}(\mathcal{H})\cong M_d(\mathbb C)$. We also define
$$ \overline{A}=\sum_{i,j=1}^{d} \overline{A_{i,j}}|e_i\rangle \langle e_j|,\quad A^*=\sum_{i,j=1}^{d} \overline{A_{j,i}}|e_i\rangle \langle e_j|$$
where $A^*=\overline{A}^{T}$.

For $A$, the Schatten-$p$ norm is defined by
$$ \|A \|_p=\left(\Tr |A^*A|^{\frac{p}{2}}\right)^{\frac{1}{p}},\quad p\geq 1.$$
The Schatten-$p$ class is
$$ \mathcal{S}_p=\{A\in \mathcal{B}(\mathcal{H})\,:\, \|A \|_p<\infty\}.$$

\section{Auxiliary results}
We first prove the equivalent form of the conjectured inequality.
\begin{proposition}
\label{prop:m-summand-equivalence}
Let $1<p<\infty$ and 
\(
 q=p/(p-1)
\). Let $\mathcal H$ be a finite-dimensional Hilbert space of arbitrary dimensions. For a constant $C>0$, the following assertions are equivalent.

\noindent {\rm (i)} For 
$A_1,\ldots,A_m\in\mathcal B(\mathcal H)$,
\begin{equation}\label{eq:m-summand-matrix}
 \left\|\sum_{k=1}^mA_k\right\|_p
 \leq
 C\left\|\sum_{k=1}^m|A_k|\right\|_p.
\end{equation}

\noindent {\rm (ii)} For 
\(
 R,X_1,\ldots,X_m\in\mathcal B(\mathcal H)_+
\), 
one has
\begin{equation}\label{eq:m-summand-positive}
 \sum_{k=1}^m\|RX_k\|_1
 \leq
 C\|R\|_q
 \left\|\sum_{k=1}^mX_k\right\|_p.
\end{equation}

\end{proposition}

\begin{proof}
Assume first that \eqref{eq:m-summand-matrix} holds. Let
\[
 R,X_1,\ldots,X_m\geq0.
\]
For each $k$, the variational characterization of the trace norm
provides a unitary $U_k$ such that
\[
 \|RX_k\|_1=
 \|X_kR\|_1=
 \Re\Tr(U_kX_kR)=
 \Re\Tr(RU_kX_k).
\]
Consequently,
\[
 \begin{aligned}
 \sum_{k=1}^m\|RX_k\|_1
 =
 \Re\Tr\left(
   R\sum_{k=1}^mU_kX_k
 \right)\leq
 \|R\|_q
 \left\|\sum_{k=1}^mU_kX_k\right\|_p.
 \end{aligned}
\]
Since each $U_k$ is unitary,
\[
 |U_kX_k|=X_k.
\]
Applying \eqref{eq:m-summand-matrix} to
\(
 A_k=U_kX_k
\) therefore gives
\[
 \left\|\sum_{k=1}^mU_kX_k\right\|_p
 \leq
 C\left\|\sum_{k=1}^mX_k\right\|_p.
\]
Combining the preceding inequalities proves
\eqref{eq:m-summand-positive}.

Conversely, assume that \eqref{eq:m-summand-positive} holds. Let
\(
 A_k=U_k|A_k|
\)
be the polar decomposition of $A_k$, where the partial
isometry $U_k$ is a contraction on $\mathcal H$.

By Schatten duality,
\[
 \left\|\sum_{k=1}^mA_k\right\|_p
 =
 \sup_{\|Z\|_q\leq1}
 \left|
   \Tr\left(
     Z^*\sum_{k=1}^mU_k|A_k|
   \right)
 \right|.
\]
Fix $Z$ with $\|Z\|_q\leq1$, and write its polar decomposition as
\[
 Z=WR,
 \qquad
 R=|Z|\geq0,
 \qquad
 \|W\|_\infty\leq1.
\]
Then
\(
 \|R\|_q=\|Z\|_q\leq1 
\). For each $k$, 
\[
 \left|\Tr(Z^*U_k|A_k|)\right|
 =
 \left|\Tr(RW^*U_k|A_k|)\right|
 \leq
 \||A_k|R\|_1\|W^*U_k\|_\infty
 \leq
 \|R|A_k|\|_1.
\]
It follows from \eqref{eq:m-summand-positive} that
\[
 \left|
   \Tr\left(
     Z^*\sum_{k=1}^mU_k|A_k|
   \right)
 \right|
\leq
 \sum_{k=1}^m\big\|R|A_k|\big\|_1\leq
 C\left\|\sum_{k=1}^m|A_k|\right\|_p.
\]
Taking the supremum over all $Z$ with $\|Z\|_q\leq1$ proves
\eqref{eq:m-summand-matrix}.
\end{proof}

The following results are the key ingredients to establish the main theorem.
\begin{lemma}\label{lem:tensor}
Let $i\in\{1,\cdots,m\}$ and $H_i\in \mathcal{B}(\mathcal{H})_+$. Set
\[
 H=\left(\sum_{i=1}^m H_i^2\right)^{1/2}.
\]
Then
\begin{equation}\label{eq:tensor}
 \sum_{i=1}^m H_i\otimes\overline{H_i}
 \leq H\otimes\overline H.
\end{equation}
\end{lemma}

\begin{proof}
It suffices to work on the support of $H$, where $H$ is
invertible.  Denote
\[
 L_i=H^{-1/2}H_iH^{-1/2},
 \qquad
 C_i=H_iH^{-1}=H^{1/2}L_iH^{-1/2}.
\]
Then
\[
 \sum_i C_i^*C_i
 =H^{-1}\left(\sum_iH_i^2\right)H^{-1}
 =\id.
\]
We now define
\begin{align*}
 \mathcal E\,:\, \mathcal{B(H)}&\to\mathcal{B(H)}
 \\  X\,\,&\mapsto \sum_i C_i X C_i^*.
\end{align*}
Hence $\mathcal E$ is a completely positive and trace preserving map on $\mathcal{B}(\mathcal{H})$.  Its adjoint is unital, completely positive and therefore has
operator norm one.  Hence the spectral radius of $\mathcal E$ is at
most one.  

Define
$$ {\rm vec}(X)=\sum X_{m,n} e_m\otimes e_n.$$
For any matrix $D$,
$$ {\rm vec}(D X D^*)=\sum_{a,b}\sum_{m,n} D_{a,m}X_{m,n}\overline{D_{b,n}}e_a\otimes e_b=(D\otimes \overline{D}){\rm vec}(X).$$
Hence the matrix representation of
$\mathcal E$ is
\[
 \sum_i C_i\otimes\overline{C_i}
\]
and it has spectral radius at most one.

On the other hand,
\begin{align*}
 \sum_i C_i\otimes\overline{C_i}
 &=
 (H^{1/2}\otimes\overline H^{\,1/2})
 \left(\sum_iL_i\otimes\overline{L_i}\right)
 (H^{-1/2}\otimes\overline H^{\,-1/2}).
\end{align*}
It is therefore similar to
\[
 M=\sum_iL_i\otimes\overline{L_i}.
\]
Each $L_i\otimes\overline{L_i}$ is positive semidefinite, so
$M\geq0$.  Similarity and the spectral-radius bound imply that all
eigenvalues of $M$ belong to $[0,1]$.  Hence $M\leq\id$.
Multiplying each side by $H^{1/2}\otimes\overline H^{\,1/2}$ gives
$$(H^{1/2}\otimes\overline H^{\,1/2} )M (H^{1/2}\otimes\overline H^{\,1/2})=\sum_{i=1}^{m}H_i\otimes \overline{H_i}\leq H\otimes \overline{H}. $$
This completes the proof.
\end{proof}

\begin{corollary}\label{cor:multi-families-tensor}
Let
\[
 E_{ak}\in\mathcal B(\mathcal H)_+,
 \qquad
 1\leq a\leq N,\quad 1\leq k\leq m.
\]
For each $k$, set
\[
 E_k
 =
 \left( E_{1k}^2+E_{2k}^2+\cdots+E_{Nk}^2\right)^{1/2}.
\]
Then
\begin{equation}\label{eq:multi-families-tensor}
 \sum_{a=1}^N
 \left(\sum_{k=1}^mE_{ak}\right)
 \otimes
 \overline{\left(\sum_{k=1}^mE_{ak}\right)}
 \leq
 \left(\sum_{k=1}^mE_k\right)
 \otimes
 \overline{\left(\sum_{k=1}^mE_k\right)}.
\end{equation}
\end{corollary}

\begin{proof}
Consider 
\(
 \mathcal K=\mathcal H^{\oplus m} 
\)
and define
\[
 H_a
 =
 E_{a1}\oplus\cdots\oplus E_{am}
 \in\mathcal B(\mathcal K)_+.
\]
Then
\[
 \left(H_1^2+\cdots+H_N^2\right)^{1/2}
 =
 E_1\oplus\cdots\oplus E_m.
\]
Applying Lemma~\ref{lem:tensor} to the family $(H_a)_{a=1}^N$
gives
\begin{equation}\label{eq:direct-sum-tensor}
 \sum_{a=1}^NH_a\otimes\overline{H_a}
 \leq
 (E_1\oplus\cdots\oplus E_m)
 \otimes
 \overline{(E_1\oplus\cdots\oplus E_m)}.
\end{equation}

Comparing each diagonal block on the two sides of \eqref{eq:direct-sum-tensor} gives, for every $1\leq k,\ell\leq m$,
\[
 \sum_{a=1}^N
 E_{ak}\otimes\overline{E_{a\ell}}
 \leq
 E_k\otimes\overline{E_\ell}.
\]
Summing over $k$ and $\ell$, one obtains
\[
 \sum_{k,\ell=1}^m\sum_{a=1}^N
 E_{ak}\otimes\overline{E_{a\ell}}
 \leq
 \sum_{k,\ell=1}^m
 E_k\otimes\overline{E_\ell}.
\]
Since
\[
 \sum_{k,\ell=1}^m
 E_{ak}\otimes\overline{E_{a\ell}}
 =
 \left(\sum_{k=1}^mE_{ak}\right)
 \otimes
 \overline{\left(\sum_{\ell=1}^mE_{a\ell}\right)},
\]
the preceding inequality is exactly
\eqref{eq:multi-families-tensor}.
\end{proof}

\begin{theorem}\label{thm:sf}
Let
\[
 E_{ak}\in\mathcal B(\mathcal H)_+,
 \qquad
 1\leq a\leq N,\quad 1\leq k\leq m.
\]
For every $2\leq p\leq \infty$,
\begin{equation}\label{eq:sf}
 \left(
   \sum_{a=1}^N
   \left\|\sum_{k=1}^mE_{ak}\right\|_p^p
 \right)^{1/p}
 \leq
 \left\|
   \sum_{k=1}^m
   \left(\sum_{a=1}^NE_{ak}^2\right)^{1/2}
 \right\|_p.
\end{equation}
\end{theorem}

\begin{proof}
For $1\leq k\leq m$, put
\[
 E_k
 =
 \left(E_{1k}^2+\cdots+E_{Nk}^2\right)^{1/2},
\]
and set
\[
 S_a=\sum_{k=1}^mE_{ak},
 \qquad
 S=\sum_{k=1}^mE_k.
\]
Since
\(
 E_{ak}^2\leq E_k^2
\), operator monotonicity of the square root gives
\(
 E_{ak}\leq E_k
\). Summing over $k$, we obtain
\begin{equation}\label{eq:Sa-less-S}
 0\leq S_a\leq S.
\end{equation}
This directly proves the assertion for $p=\infty$.  For $p\neq \infty$, we may restrict all operators to the support of $S$. Define
\[
 K_a=S^{-1/2}S_a^{1/2},
\]
and consider the linear map
\[
\begin{aligned}
 \Phi:\mathcal S_\infty
 &\longrightarrow
 \ell_\infty(\mathcal S_\infty),
 \\
 T&\longmapsto \bigl(K_a^*TK_a\bigr)_{a=1}^N.
\end{aligned}
\]
By \eqref{eq:Sa-less-S},
\[
 K_aK_a^*
 =
 S^{-1/2}S_aS^{-1/2}
 \leq\id.
\]
Hence $\|K_a\|_\infty\leq1$, and therefore
\begin{equation}\label{eq:multi-infty-contraction}
 \|\Phi(T)\|_{\ell_\infty(\mathcal S_\infty)}
 =
 \max_a\|K_a^*TK_a\|_\infty
 \leq
 \|T\|_\infty.
\end{equation}

Set
\[
 M_a
 =
 K_aK_a^*
 =
 S^{-1/2}S_aS^{-1/2}.
\]
By Corollary~\ref{cor:multi-families-tensor},
\[
 \sum_{a=1}^NS_a\otimes\overline{S_a}
 \leq
 S\otimes\overline S.
\]
Multiplying each side with
\(
 S^{-1/2}\otimes\overline S^{\,-1/2}
\)
gives
\[
 \sum_{a=1}^NM_a\otimes\overline{M_a}
 \leq
 \id.
\]
This yields
\begin{equation} \label{eq:multi-two-contraction}
\begin{aligned}
 \|\Phi(T)\|_{\ell_2(\mathcal S_2)}^2
 &=
 \sum_{a=1}^N\|K_a^*TK_a\|_2^2
\\
 &=
 \sum_{a=1}^N
 \Tr(T^*M_aTM_a)\\
 &=
 \left\langle
   \operatorname{vec}(T),
   \left(
     \sum_{a=1}^NM_a\otimes\overline{M_a}
   \right)
   \operatorname{vec}(T)
 \right\rangle
\\
 &\leq
 \|T\|_2^2.
\end{aligned}
\end{equation}
Thus
\[
 \Phi:\mathcal S_2\longrightarrow\ell_2(\mathcal S_2),\quad \Phi:\mathcal S_\infty
 \longrightarrow
 \ell_\infty(\mathcal S_\infty)
\]
are both contractions. Note that
$$ \ell_p(\mathcal{S}_p)=L_p(\mathcal{N}),\quad \mathcal N=\mathcal{B}(\mathcal{H})^{\oplus N}$$
where $\mathcal{N}$ is a finite von Neumann algebra. Complex interpolation gives
\[
 [\mathcal S_2,\mathcal S_\infty]_\theta
 =
 \mathcal S_p,\qquad
 [
   \ell_2(\mathcal S_2),
   \ell_\infty(\mathcal S_\infty)
 ]_\theta
 =
 \ell_p(\mathcal S_p),
 \qquad
 \theta=1-\frac{2}{p}.
\]
Interpolating
\eqref{eq:multi-infty-contraction} and
\eqref{eq:multi-two-contraction}, we obtain
\[
 \|\Phi(T)\|_p=\left(
   \sum_{a=1}^N
   \|K_a^*TK_a\|_p^p
 \right)^{1/p}
 \leq
 \|T\|_p.
\]
Finally,
\[
 K_a^*SK_a
 =
 S_a^{1/2}S^{-1/2}SS^{-1/2}S_a^{1/2}
 =
 S_a.
\]
Taking $T=S=E_1+\cdots+E_m$ proves the desired result.
\end{proof}

\section{Proof of the main theorem}
We first prove the rank-$m$ result.

\begin{lemma}\label{lem:rank-m}
Let $2\leq p\leq \infty$ and $q=p/(p-1)$.
Let $X_1,\ldots,X_m,R\in\mathcal B(\mathcal H)_+$, and suppose that
\[
 S:=\sum_{k=1}^mX_k,
 \qquad
 \rank S\leq m.
\]
Then
\[
 \sum_{k=1}^m\|RX_k\|_1
 \leq
 C_{p,m}\|R\|_q\|S\|_p.
\]
Consequently, for $A_1,\cdots,A_m \in \mathcal{B}(\mathcal{H})$ where $|A_k|=X_k$,
\begin{equation*}
 \left\|\sum_{k=1}^mA_k\right\|_p
 \leq
 C_{p,m}\left\|\sum_{k=1}^m|A_k|\right\|_p.
\end{equation*}
\end{lemma}

\begin{proof}
For positive semidefinite $R$ and $X$, Cauchy--Schwarz inequality gives
\[
 \begin{aligned}
 \|RX\|_1=
 \|RX^{1/2}X^{1/2}\|_1\leq
 \|RX^{1/2}\|_2\|X^{1/2}\|_2=
 \sqrt{\Tr(R^2X)\Tr X}.
 \end{aligned}
\]
Applying this estimate to every $X_k$ and then using scalar
Cauchy--Schwarz inequality, we obtain
\begin{equation}\label{eq:c-s}
\begin{aligned}
 \sum_{k=1}^m\|RX_k\|_1
 &\leq
 \sum_{k=1}^m
 \sqrt{\Tr(R^2X_k)\Tr X_k}\\
 &\leq
 \sqrt{
   \left(\sum_{k=1}^m\Tr(R^2X_k)\right)
   \left(\sum_{k=1}^m\Tr X_k\right)
 }\\
 &=
 \sqrt{\Tr(R^2S)\Tr S}.
\end{aligned}
\end{equation}

Let
\[
 \|S\|_\infty=s_1\geq s_2\geq\cdots\geq s_m\geq0
\]
be the eigenvalues of $S$, padded with zeros if necessary.
Since $S\leq s_1\id$ and $1\leq q\leq2$,
\[
 \Tr(R^2S)
 \leq
 s_1\Tr R^2
 =
 s_1\|R\|_2^2
 \leq
 s_1\|R\|_q^2.
\]
It follows that
\begin{equation}\label{eq:rank-r-intermediate}
 \sum_{k=1}^m\|RX_k\|_1
 \leq
 \|R\|_q\sqrt{s_1\Tr S}.
\end{equation}

First let $p\neq \infty$. We are going to check that
\begin{equation}\label{eq:rank-r-scalar}
 \sqrt{s_1 \Tr S}=\sqrt{s_1(s_1+\cdots+s_m)}
 \leq
 C_{p,m}
 \left(\sum_{j=1}^m s_j^p\right)^{1/p}.
\end{equation}
If
$s_2=\cdots=s_m=0$, it holds trivially. Otherwise, put
\[
 \beta
 =
 \frac{s_2+\cdots+s_m}{m-1},
 \qquad
 x=\frac{s_1}{\beta}.
\]
Since $s_j\leq s_1$ for every $j$, we have $x\geq1$. Convexity of
$t\mapsto t^p$ gives
\[
 s_2^p+\cdots+s_m^p
 \geq
 (m-1)\beta^p.
\]
Hence we obtain
$$s_1^p+s_2^p+\cdots+s_m^p
 \geq
 (x\beta)^p+ (m-1)\beta^p=\beta^p(x^p+m-1).$$ 
 Consequently,
\[
 \frac{
   \sqrt{s_1(s_1+\cdots+s_m)}
 }{
   (s_1^p+\cdots+s_m^p)^{1/p}
 }
 \leq \frac{\sqrt{x(x+m-1)}}
 {(x^p+m-1)^{1/p}}=
 h_{p,m}(x).
\]
A direct logarithmic differentiation gives
\[
 \frac{d}{dx}\ln h_{p,m}(x)
 =
 \frac{
   (m-1)(2x+m-1-x^p)
 }{
   2x(x+m-1)(x^p+m-1)
 }.
\]
Thus $h_{p,m}$ attains its maximum at the point $x_{p,m}\in (1,\infty)$ where
\[
 x_{p,m}^{\,p}=2x_{p,m}+m-1.
\]
Hence
\[
 \max_{x\geq1}h_{p,m}(x)
 =
 h_{p,m}(x_{p,m})
 =
 C_{p,m},
\]
which proves \eqref{eq:rank-r-scalar}. 

For $p=\infty$, we use
\(
 \Tr S\leq m s_1
\)
in \eqref{eq:rank-r-intermediate} with $q=1$, and obtain
\[
 \sum_{k=1}^m\|RX_k\|_1
 \leq
 \sqrt m\,\|R\|_1\|S\|_\infty.
\]
Combining it with \eqref{eq:rank-r-scalar} gives
\[ \sum_{k=1}^m\|RX_k\|_1
 \leq
 C_{p,m}\|R\|_q\|S\|_p\]
for $2\leq p\leq \infty$. This completes the proof for the first assertion.

To prove the second assertion, take a matrix $W$ with
$\|W\|_q\leq1$ and write its polar decomposition as $W=VR$, where
$R=|W|\geq0$ and $V$ is a partial isometry.  Also write the polar decomposition 
$A_k=U_k|A_k|=U_k X_k$. Then
\begin{align*}
 \bigl|\Tr W^*(U_1|A_1|+\cdots+U_m|A_m|)\bigr|
 &\leq \|R\,|A_1|\|_1+\cdots+\|R\,|A_m|\|_1\\
 &\leq C_{p,m}\|R\|_q\||A_1|+\cdots|A_m|\|_p\\
 &\leq C_{p,m}\||A_1|+\cdots|A_m|\|_p.
\end{align*}
Schatten duality then proves the second assertion.
\end{proof}

Now we prove the main theorem.
\begin{proof}[Proof of Theorem~\ref{thm:main}]
Write the polar decompositions
\(
 A_k=U_k|A_k|
\), extend each $U_k$ to unitary when necessary,
and put
\[
 Z=A_1+\cdots +A_m=U_1|A_1|+\cdots+U_m|A_m|.
\]
Choose a complete orthonormal family of left singular vectors $u_i$ of
$Z$, let $z_i$ be the corresponding singular values, and set
\[
 Q_i=|u_i\rangle\langle u_i|.
\]
Define
\[
 X_{ik}=|Q_iA_k|=(|A_k|U_k^*Q_iU_k|A_k|)^{1/2},
\]

The matrix $Q_iZ$ is at most rank-one and its only singular
value is $z_i$.  Let
\[
 Q_iA_k=\widetilde U_{ik}X_{ik}
\]
be polar decompositions.  Since
\(
 \rank X_{ik}\leq1
\), 
we have 
$$\rank(X_{i1}+\cdots X_{im})\leq m.$$  
Then, applying Lemma~\ref{lem:rank-m} gives
\[
 z_i=\|Q_iZ\|_p
 =\left\|\sum_{k=1}^{m}Q_iA_k\right\|_p=\left\|\sum_{k=1}^{m}\widetilde{U}_{ik}X_{ik}\right\|_p\leq C_{p,m}\left\|X_{i1}+\cdots+X_{im}\right\|_p.
\]
Summing over all singular values we obtain
\begin{align*}
 \sum_i z_i^p=\|A_1+\cdots+A_m\|_p^p
 \leq C_{p,m}^p\sum_i\|X_{i1}+\cdots+X_{im}\|_p^p.
\end{align*}

Besides, the identity $\sum_iQ_i=\id$ yields, for each $k$,
\[
 \sum_i X_{ik}^2
 =|A_k|U_k^*\left(\sum_iQ_i\right)U_k|A_k|
 =|A_k|^2.
\]
Then Theorem~\ref{thm:sf} gives
\begin{align*}
 \sum_i \left\|\sum_{k=1}^{m}X_{ik} \right\|_p^p\leq \left\|
   \sum_{k=1}^{m}\left(\sum_iX_{ik}^2\right)^{1/2}
 \right\|_p^p=\||A_1|+\cdots+|A_m|\|_p^p.
\end{align*}
This proves \eqref{eq:main}. 

% It remains to verify sharpness.  Choose unit vectors $x,y$ such that
% \[
%  \langle x,y\rangle
%  =c_p:=\frac{x_p-1}{x_p+1},
% \]
% and a unit vector $\alpha$.  Let
% \[
%  A=|\alpha\rangle\langle x|,
%  \qquad
%  B=|\alpha\rangle\langle y|.
% \]
% Then $A+B=|\alpha\rangle\langle x+y|$, whereas
% $|A|+|B|=|x\rangle\langle x|+|y\rangle\langle y|$ has nonzero
% eigenvalues $1+c_p$ and $1-c_p$.  Hence
% \[
%  \frac{\|A+B\|_p}{\||A|+|B|\|_p}
%  =
%  \frac{\sqrt{2(1+c_p)}}
%  {\bigl((1+c_p)^p+(1-c_p)^p\bigr)^{1/p}}
%  =
%  \frac{\sqrt{x_p(x_p+1)}}{(x_p^p+1)^{1/p}}
%  =C_p.
% \]
% Thus the constant is optimal.  Letting $p\to\infty$ gives
% $C_{\infty,m}=\sqrt{m}$ and finishes the proof.

It remains to verify sharpness. Choose ${\rm dim}(\mathcal H)\geq m$. For $2\leq p< \infty$, set
\[
 c_{p,m}
 =
 \frac{x_{p,m}-1}{x_{p,m}+m-1}.
\]
Choose unit vectors $v_1,\ldots,v_m\in\mathcal H$ where
\[
 \langle v_j,v_k\rangle
 =
 \begin{cases}
  1,&j=k,\\
  c_{p,m},&j\neq k.
 \end{cases}
\]
Such vectors exist because the related Gram matrix
is positive definite. Let $\alpha\in\mathcal H$ be a unit vector and
define
\[
 A_k=|\alpha\rangle\langle v_k|,
 \qquad
 1\leq k\leq m.
\]
Then
\[
 \left\|\sum_{k=1}^mA_k\right\|_p
 =
 \left\|\sum_{k=1}^mv_k\right\|_{\mathcal{H}}
 =
 \sqrt{m\bigl(1+(m-1)c_{p,m}\bigr)}.
\]

On the other hand, \(
 |A_k|=|v_k\rangle\langle v_k|
\). Hence
\[
 \sum_{k=1}^m|A_k|
 =
 \sum_{k=1}^m|v_k\rangle\langle v_k|.
\]
Its nonzero eigenvalues are
\[
 1+(m-1)c_{p,m},\,
 1-c_{p,m},\,\cdots,\,1-c_{p,m}
\]
where $1-c_{p,m}$ repeats $m-1$ times. Therefore
\[
 \left\|\sum_{k=1}^m|A_k|\right\|_p
 =
 \left(
   \bigl(1+(m-1)c_{p,m}\bigr)^p
   +(m-1)(1-c_{p,m})^p
 \right)^{1/p}.
\]
Consequently,
\begin{align*}
 \frac{
   \left\|\sum_{k=1}^mA_k\right\|_p
 }{
   \left\|\sum_{k=1}^m|A_k|\right\|_p
 }
 &=
 \frac{
   \sqrt{m\bigl(1+(m-1)c_{p,m}\bigr)}
 }{
   \left(
     \bigl(1+(m-1)c_{p,m}\bigr)^p
     +(m-1)(1-c_{p,m})^p
   \right)^{1/p}
 }.
\end{align*}
The right-hand side is exactly
\begin{align*}
 \frac{\sqrt{x_{p,m}(x_{p,m}+m-1)}}
 {(x_{p,m}^p+m-1)^{1/p}}=
 C_{p,m}.
\end{align*}

At the $p=\infty$ case, choose
$v_1,\ldots,v_m$ to be orthonormal. Then
\[
 \left\|\sum_{k=1}^mA_k\right\|_\infty
 =
 \sqrt m,
 \qquad
 \left\|\sum_{k=1}^m|A_k|\right\|_\infty
 =
 1.
\]
This completes the proof.
\end{proof}

By Proposition~\ref{prop:m-summand-equivalence}, the result has the following equivalent quasi-reverse Minkowski form:
\begin{corollary}\label{cor:positive}
Let $2\leq p\leq \infty$, $q=p/(p-1)$. Let $R,X_1,\cdots,X_m\in \mathcal{B}(\mathcal{H})_+$.  Then
\[
 \|RX_1\|_1+\cdots+\|RX_m\|_1
 \leq C_{p,m}\|R\|_q\|X_1+\cdots+X_m\|_p.
\]
The constant $C_{p,m}$ is optimal over all $\dim(\mathcal{H})<\infty$.
\end{corollary}

\section{Counterexamples for \texorpdfstring{$1<p<2$}{1 < p < 2}}
\label{sec:below-two}
In this section, we only discuss the $2$-summand case and simply write $C_p=C_{p,2},\, x_p=x_{p,m}$. 

We first present a numerical counterexample with $p = 3/2$. Let
\begin{align}\label{eq:numeric-ce}
 A=\begin{pmatrix}1&0\\0&0\end{pmatrix},
 \qquad
 B=\begin{pmatrix}
       72/125&96/125\\
       21/125&28/125
    \end{pmatrix}.
  \end{align}
Both matrices have rank one.  The eigenvalues of $|A|+|B|$ are
$8/5,2/5$, and the squared singular values of $A+B$ are
$392/125,2/125$.  Hence
\begin{equation*}
 \frac{\|A+B\|_{3/2}}{\bigl\||A|+|B|\bigr\|_{3/2}}
 =
 \frac{
  \left[(392/125)^{3/4}+(2/125)^{3/4}\right]^{2/3}}
 {\left[(8/5)^{3/2}+(2/5)^{3/2}\right]^{2/3}}
 \approx1.036194814532626.
\end{equation*}
The solution of $x^{3/2}=2x+1$ is
$x\approx4.864536512317585$, so the proposed value is only
\[
 C_{3/2}
 =\frac{\sqrt{x(x+1)}}{(x^{3/2}+1)^{2/3}}
 \approx1.034653951851434.
\]
Thus \eqref{eq:numeric-ce} is a strict counterexample.

In fact, let $A=|\alpha\rangle\langle\xi|$ and
$B=|\beta\rangle\langle\eta|$, where all four vectors are real unit
vectors and
\[
 c=\langle\xi,\eta\rangle\in[0,1],
 \qquad
 d=\langle\alpha,\beta\rangle\in[0,1].
\]
Then one can compute
\begin{equation*}
 \frac{\|A+B\|_p}{\bigl\||A|+|B|\bigr\|_p}=
 \frac{
  \left((1+c)^{p/2}(1+d)^{p/2}
       +(1-c)^{p/2}(1-d)^{p/2}\right)^{1/p}}
 {\left((1+c)^p+(1-c)^p\right)^{1/p}}=F_p(c,d).
\end{equation*}
Put $r=p/2$. For fixed $c\in(0,1)$, one can check that
$$\partial_d F_p(c,d)|_{d=0}>0,\quad \lim_{d\uparrow 1}\partial_d F_p(c,d)\to-\infty.$$
Hence for fixed $c$, the unique
maximizer of $F_p(c,\cdot)$ is interior and is achieved at 
\begin{equation*}
 \tilde{d}_c=\frac{(\frac{1+c}{1-c})^{p/(2-p)}-1}{(\frac{1+c}{1-c})^{p/(2-p)}+1}.
\end{equation*}
Now let $x_p>1$ solve $x_p^p=2x_p+1$ and choose
\[
 c=\frac{x_p-1}{x_p+1},
 \qquad d=\tilde{d}_c.
\]
Then one obtains 
\[
 \max_{c} F_p(c,d)>\max_{c} F_p(c,1)=C_p.
\]
Real unit vectors with inner products $c$ and $d$ realize this
quotient.

We conclude the unsolved question as following:
\begin{question}\label{open}
For every $1<p<2$ and $A_1,\cdots,A_m\in \mathcal{B(H)}$. What is the sharp constant $\widehat C_{p,m}$ that makes
\[
 \|A_1+\cdots+A_m\|_p
 \leq\widehat C_{p,m}\bigl\||A_1|+\cdots+|A_m|\bigr\|_p\,?
\]
\end{question}

Moreover, in the $2$-summand case, we point out that the maximum of $F_p(c,d)$ is still not the
correct candidate for the sharp constant.  Indeed, for $p=6/5$, let
\[
 A=\begin{pmatrix}
  1&0&0\\0&0&0\\0&0&0
 \end{pmatrix},
 \qquad
 B=\begin{pmatrix}
  0.68128169&0.59021564&-0.00149331\\
  0.32727794&0.28426802&0\\
  0&0&0
 \end{pmatrix}.
\]
A direct computation gives
\[
  \frac{\|A+B\|_{6/5}}{\bigl\||A|+|B|\bigr\|_{6/5}}
\approx1.00733342639935,
\]
while
\[
  \max_{0\leq c,d\leq 1}F_{6/5}(c,d)
  \approx 1.00732496576788.
\]
As a result, the universal sharp constant in Question~\ref{open} might not be achieved at a specific case and the extremal value may increase as the dimension of the ambient Hilbert space grows.

\section*{}

\noindent{\bf Acknowledgement.} The research of Hongsen Qiu is supported by the National Natural Science Foundation of China (Grant No. 12371138 and No. W2441002).  The author is grateful to Haonan Zhang for offering valuable suggestions, which improves the result of this paper. The author would also like to thank GPT-5.6-Sol for giving the two numerical couterexamples, language polishing and TeX editing.

\end{document}